\documentclass[a4paper,fleqn]{cas-sc}

\usepackage[numbers,sort]{natbib}
 \usepackage{longtable}
\usepackage{float}
\def\tsc#1{\csdef{#1}{\textsc{\lowercase{#1}}\xspace}}
\DeclareMathOperator*{\fl}{fl}
\DeclareMathOperator*{\MIXEDDjD}{MAD(j)}
\DeclareMathOperator*{\MIXEDSINGLE}{MAS(j)}
\tsc{WGM}
\tsc{QE}
\tsc{EP}
\tsc{PMS}
\tsc{BEC}
\tsc{DE}
 \newtheorem{theorem}{Theorem}
 
 \newdefinition{rmk}{Remark}
 \newproof{pf}{Proof}
 \newproof{pot}{Proof of Theorem \ref{thm2}}

\begin{document}
\let\WriteBookmarks\relax
\def\floatpagepagefraction{1}
\def\textpagefraction{.001}
\def\dt{h}
\shorttitle{Mixed precision numerical methods for ODE}
\shortauthors{M. Al Sayed Ali et~al.}
%\begin{frontmatter}

\title [mode = title]{Mixed precision  explicit numerical methods for ordinary differential equations}
\tnotemark[1]

\author[1]{M. {Al Sayed Ali}}[type=editor,
                        auid=,bioid=,
                        prefix=,
                        role=,
                        ]
\ead{mouhamad.alsayedali@univ-rennes.fr}
\affiliation[1]{organization={IRMAR, Univ. Rennes},
                city={Rennes},
%               citysep={}, % Uncomment if no comma needed between city and postcode
                postcode={35000},
                country={France}}

\author[2]{S. Bernard}[type=editor,
                        auid=,bioid=,
                        prefix=,
                        role=,
                        ]
\ead{samuel.bernard@inria.fr}
\affiliation[2]{organization={Universit\'e Claude Bernard Lyon 1, CNRS, École Centrale de Lyon, INSA Lyon, Universit\'e Jean Monnet, ICJ UMR5208, Inria},
                city={Villeurbanne},
%               citysep={}, % Uncomment if no comma needed between city and postcode
                postcode={69622},
                country={France}}

\author[3]{A. Marzorati}[type=editor,
                        auid=,bioid=,
                        prefix=,
                        role=,
			]
\ead{arsene.marzorati@inria.fr}
\affiliation[3]{organization={Inria},
                city={Lyon},
%               citysep={}, % Uncomment if no comma needed between city and postcode
                country={France}}

\author[4]{J. Rouzaud-Cornabas}[type=editor,
                        auid=,bioid=,
                        prefix=,
                        role=,
                        ]
\ead{jonathan.rouzaud-cornabas@inria.fr}
\affiliation[4]{organization={CITI, INSA Lyon, CNRS, Inria, LIRIS UMR5205,  Universit\'e Claude Bernard Lyon 1, ECL, Universit\'e Lumi\`ere Lyon 2},
                city={Lyon},
%               citysep={}, % Uncomment if no comma needed between city and postcode
                country={France}}

\cortext[cor1]{Corresponding author}

% REQUIRED
\begin{abstract}
Our objective is to solve large systems of ordinary differential equations (ODEs) commonly used to model biological processes. These equations are typically nonlinear, complex, and high-dimensional. In computational biology, such ODEs are generally solved using numerical methods. In this work, we focus on explicit numerical methods because of their flexibility. However, their limited stability regions may result in high computational costs. To mitigate this issue, we investigate mixed precision algorithms designed to reduce computational effort by performing selected parts of the numerical method in lower arithmetic precision. We develop several mixed precision explicit methods and assess their performance on two large-scale biological benchmark ODE models. Our theoretical analysis highlights the effectiveness of partially reducing arithmetic precision within explicit methods. Numerical experiments demonstrate that our mixed methods—implemented in both sequential and parallel versions using MPI—combining single (float) and double precision arithmetic can achieve up to twice the speed of a fully double precision implementation while preserving the same level of accuracy. Furthermore, the results indicate that decreasing the timestep improves the performance and robustness of our mixed methods, while the single precision method fails to converge.
\end{abstract}

% REQUIRED

\maketitle

\section{Introduction}

 On modern architectures, the performance of single precision (float, usually occupying $32$ bits in memory) operations is often at least twice as fast as the performance of double precision ($64$ bits) operations (\cite{AA2021, HM2022}). Because single precision is limited in accuracy, double precision has become the {\it de facto} standard in scientific computing. Double precision algorithms are less susceptible to numerical instabilities at the expense of higher computational cost. Lowering the arithmetic precision could speed up computations and communications without compromising accuracy.  Lower precision ({\it e.g}., single precision) algorithms may be employed; however, this typically comes at the cost of reduced accuracy or may even lead to numerical divergence due to stability issues. Mixed precision algorithms, which combine lower and higher arithmetic precisions, could therefore be used to increase performances while maintaining
high accuracy. Mixed precision algorithms have become popular in numerical linear algebra \cite{AA2021,ALSGK2019,ZFZI2022}, machine learning \cite{DMMK2018,MSDK2019}, climate and weather model simulation \cite{ADP2022,KDP2020,KDP2021,PCKSP2022,VDLSC2017,DSDP2017} and for numerical integration \cite{BGGH2021,G2022,HMR2008}.

 In this paper, we are interested in using mixed precision in explicit numerical methods for solving large systems of ordinary differential equations (ODEs). Our approach work for all size of ODE, but we are mainly interested in the ODEs obtained from biological models that are of large size (see \cite{EBE2014, EBE2017, LB2015}).  To our knowledge, few authors \cite{CS2022} have studied mixed precision in an explicit numerical method.  All other authors  \cite{K2022,BGGH2021,G2022,BRG2023}
have used mixed precision in implicit numerical methods. In \cite{CS2022}, the authors analyzed the accuracy and the stability of a designed mixed precision explicit Runge-Kutta-Chebyshev (RKC).  They showed that they can preserve the order $p$ of RKC by $p$ higher precision evaluations of the right-hand side (RHS) of the ODE at each integration step. Therefore, this approach can only be applied when the number of stages is strictly greater than the order of the method. Otherwise, we have to compute all the numerical method with a higher precision.

In our previous work~\cite{ABMC2025}, we investigated the use of mixed precision arithmetic within an implicit numerical method to accelerate the computation of the solution of large-scale ordinary differential equations (ODEs). As each iteration of an implicit numerical method involves the solution of a large nonlinear system—typically handled via Newton's method—the algorithm presents multiple opportunities to employ reduced precision in selected components to improve computational efficiency. In particular, at every time step, Newton's method requires the solution of a linear system. The computational expense of these solves, together with their limited parallel scalability, can become a performance bottleneck, even in light of the superior stability properties characteristic of implicit methods.

Explicit numerical methods, particularly when implemented with MPI, generally exhibit better scalability than implicit methods, as they primarily rely on local computations and require relatively limited communication. By avoiding the computation of the solution of linear and nonlinear systems, they significantly reduce synchronization and communication overhead. As a result, the opportunities for applying reduced precision are more constrained and are mainly associated with the evaluation of the method's stages, that is, the computation of the ODE right-hand side.

In this paper, we show that, despite these limitations, it is still possible to accelerate the computation of explicit numerical methods by selectively lowering the arithmetic precision in some computations of the stages of these methods.

This mixed precision strategy reduces computational time and memory usage while controlling the loss of numerical accuracy. Numerical experiments show that the proposed mixed precision methods, whether implemented in sequential or parallel environments, achieve speedups of up to a factor of $2.5$ compared with full double precision implementations, while maintaining comparable accuracy. Although fully single precision implementations can also yield speedups of up to a factor of $2.5$, they generally fail to deliver sufficient accuracy.

Finally, we show both theoretically and numerically that, as the time step tends to zero, the mixed precision solution converges to the double precision solution, whereas the single precision solution diverges from it.

%In \cite{LWFY2011}, they constructed a parallel four-order explicit Runge-Kutta (RK) method, that has high accuracy and computational efficiency. They computed an approximation of the stages $k_l$  %from those computed at the previous iterations. The new stages $k_l$ become independent and so they can be computed in parallel. This approach is only efficient for small sizes of ordinary %differential equations, and the theoretical results are only valid
%for  a $1-D$ ordinary differential equation.

This paper is organized as follows: In Section \ref{sec:Exp}, we define our mixed precision  methods, then we provide the necessary materials, including numerical methods and their mixed versions, as well as two biological models. In Section \ref{Sec:results}, we present theoretical and numerical results demonstrating the efficiency of our mixed approaches. Finally, we give some conclusions in Section \ref{Sec:Con}.

 \section{Materials and Methods}\label{sec:Exp}

Consider the  system of  ODEs
\begin{equation}\label{eqo1}
\dot{y}(t) =f(t,y(t)), \quad  t_0 \leq t \leq T, \quad  y(t_0) =
y_{0},
\end{equation}

where
%$f:[t_0,T]\times \R^{n}\longrightarrow \R^{n}$ is a smooth function
$y(t) \in  \mathbb{R}^{n}$.

Most of these systems of ODE arise either from the spatial discretization of partial differential equations—using finite difference, finite element, or finite volume methods—or from the direct modeling of physical and biological phenomena.

 A class of explicit numerical methods for solving  (\ref{eqo1}) is given by
 \begin{equation}\label{E_S1}
  y_{i+1}=y_{i}+h\displaystyle \sum_{l=1}^{q}a_{l}k_{l,i},\ \ i=0,\ldots,N-1
\end{equation}
where $k_{l,i}=f(t_{_{l,i}},y_{_{l,i}})$, $y_{i}$ is an approximation to $y(t_{i})$  with
$t_{i}=t_{0}+ih,\ h=\frac{T-t_{0}}{N},\ t_{_{l,i}} \in [t_{0},T], y_{l,i}=F(y_{i-q},\ldots,y_{i},k_{1,i},\ldots,k_{l-1,i})$ depends on the previous iterations $y_i$, the previous stages $k_{j,i},\ j=1,\ldots,l-1$, and the function $f.$  Here, we assume that the timestep $h$ is constant, but the theory of our mixed methods remains working for variable timestep. We also assume that $y_{l,i}$ can be expressed as follows:
\begin{equation}\label{E_y_li}
y_{l,i}=\sum_{j=0}^{q} \alpha_{l,j} y_{i-j}+h\sum_{j=1}^{l-1} \beta_{l,j} k_{j,i}
\end{equation}
where $\alpha_{l,j}, \beta_{l,j}$ are constants. Additionally, we assume that the numerical method  $(y_i)$ is of order $p$, such that $\|y_i-y(t_i)\|\le C_{y} \dt^p,$
where $C_y$ is a constant.

 In the next section, we will see that most standard explicit methods such as Runge-Kutta, Adams-Bashforth can be written as in (\ref{E_S1}) that satisfy (\ref{E_y_li}).

Each iteration of the method described in equation (\ref{E_S1}) requires $q$ evaluations of the function $f$. In our models, each evaluation requires $\mathcal{O}(n^2)$ operations which is too expensive since $n$ is too large. In order to reduce the computational cost, we compute some of these evaluations in a lower precision. With this approach of lowering precision, the numerical methods will be called mixed methods. The formulas of these mixed methods are  defined in the following section.
\subsection{Mixed precision explicit numerical methods}\label{Sec:tests}
  %  In this section, we give the algorithmic aspect of our mixed methods (\ref{E_S1_L}), the numerical methods and two biological  ODE models that will be used in section \ref{Sec:results} to numerically show the efficiency of our mixed methods.
  The construction of the mixed methods given in this section, works with any two precisions (lower and higher). In order to simplify the notation, we will use single precision and double precision. But any couples of precision would have worked.
   To define our mixed methods, we introduce the function $\text{fl}(y)$ as the evaluation of $y$, in a single precision, with a tolerance of $\varepsilon$. It is assumed that this function follows the error model $\text{fl}(y) = (I + \delta)y$, where $I$ is the identity mapping and $\|\delta\| \leq \varepsilon$ (that means $\|\text{fl}(y)-y\| \le \varepsilon\|y\|$). Here, $\varepsilon$ represents the machine epsilon, which depends on the precision being used, and $\|\cdot\|$ denotes the Euclidean norm on $\mathbb{R}^n$. As we are using single precision,
   the machine epsilon is approximately $10^{-7}.$

Our mixed methods are defined, using the function $\text{fl}$, by performing some stages $k_{l_{m},i},\ m\in \{1,\ldots,r\},\ r\le q$, with a single precision. Therefore, the numerical method $y_i$ will be replaced
by the so-called mixed method $\tilde{y}_i$ that is given by
   \begin{equation}\label{E_S1_L}
 \tilde{y}_{i+1}=w_i+h\displaystyle \sum_{\underset {{l\notin \{l_{1},\ldots,l_{r}\} }}{l=1}}^{q}a_{l}k_{l,i}+h \displaystyle \sum_{\underset {{l\in \{l_{1},\ldots,l_{r}\} }}{l=1}}^{q}a_{l}\tilde{k}_{l,i}, \quad \tilde{y}_{0}=y_0
\end{equation}
where, $w_i$ is either $\tilde{y}_i$ or $\text{fl}(\tilde{y}_i)$, and for $\ l\in \{l_1,\ldots,l_r\},$ we have
 \begin{eqnarray*}
 {\tilde{k}_{l,i}}=f^{(L)}\left(t_{_{l,i}},\displaystyle \sum_{j=0}^{q} \alpha_{l,j} \tilde{y}_{i-j}+h\sum_{j=1}^{l-1} \beta_{l,j} \tilde{k}_{j,i}\right),\end{eqnarray*}
 with $f^{(L)}(t,x):=\fl(f(\fl(t),\fl(x)))$ is obtained by implementing the function $f$ directly in single precision, rather than by simply casting double precision results to single precision. All remaining operations in (\ref{E_S1_L}) will be performed in double precision. The mixed numerical methods (\ref{E_S1_L}) will be denoted as {\it P-A$_{1}$A$_{2}$ $\ldots$ A$_{q}$}, where   $P$ can be either {\it S} or {\it D}, $A_{l}={\it S}$  when $l\in \{l_{1},\ldots,l_{r}\}$ and $A_{l}={\it D}$ otherwise. When  {\it P} $=$ {\it D} we set $w_i=\tilde{y}_i$,
 and when  {\it P} $=$ {\it S}, we set $w_i=\fl(\tilde{y}_i).$
 Therefore, when {\it P} $=$ {\it D}  the Mixed method
 is Accumulated in Double precision (MAD), and when  {\it P} $=$ {\it S} the Mixed method is Accumulated in Single precision (MAS).

When  ${\it P} = A_{k}={\it D}, k=1,\ldots,q,$ the mixed method is the  DOUBLE method. If all operations in the numerical method are performed in single precision then it will be denoted by SINGLE. We denote by MAD(j), $j\in \mathbb{N}, j\le q$, the set of mixed methods of MAD, where only $j$ stages are computed in double precision ($r=q-j$), and the remaining stages are computed in single precision. For $j=q$ we get the DOUBLE method.

We denote by MAS(j), $j\in \mathbb{N}, j\le q$, the set of mixed methods  MAS where only $j$ stages are computed in double precision ($r=q-j$), and the remaining stages are computed in single precision. In the numerical method we will only use MAS(q), which means that all stages in (\ref{E_S1_L}) are performed in double precision, and  $w_i$ is computed in a single precision. This method requires slightly more computational time than the DOUBLE method, as it performs the same operations while additionally incurring the cost of casting $w_i$ to single precision. The purpose of introducing the mixed method MAS(q) is to demonstrate the importance of evaluating the first term $w_i$ in double precision.

The numerical methods we will use in our numerical tests are :  Runge-Kutta $2$ (RK2), Runge-Kutta $4$ (RK4), Adams-Bashforth of order $1$ (AB1), and Adams-Bashforth of order $2$ (AB2). See  \cite{B2003,B1996,HNW87,HW96,VS97} for more details about these methods.
 For the Runge–Kutta methods, each time step requires multiple evaluations of the right-hand side (RHS) function: RK2 and RK4 involve $2$ and $4$ RHS evaluations per step, respectively, due to the strong interdependence of their intermediate stages. In contrast, the Adams–Bashforth methods require only a single RHS evaluation per time step, since they are explicit multistep methods that reuse information from previous steps rather than computing multiple dependent stages within the same step. These methods are written as in (\ref{E_S1}), where
\begin{eqnarray*}
q=2,~~ a_1=a_2=\frac{1}{2},~~
t_{1,i} = t_{i},~~y_{1,i}={y}_{i} , ~~
t_{2,i}=t_{i+1},~~y_{2,i}={y}_{i}+h k_{1,i},
\end{eqnarray*}
 for Runge-Kutta 2,
\begin{eqnarray*}
q=4,~~ a_1=a_4=\frac{1}{6},~~ a_2=a_3=\frac{2}{6},\nonumber\\
t_{1,i} = t_{i},~~y_{1,i}={y}_{i}, ~~ t_{2,i}=t_{i}+ h/2,~~ y_{2,i}={y}_{i}+h k_{1,i}/2\nonumber\\
t_{3,i} = t_{i}+ h/2,~~ y_{3,i}={y}_{i}+h k_{2,i}/2, ~~ t_{4,i}=t_{i+1}, ~~ y_{4,i}={y}_{i}+h k_{3,i}\nonumber
\end{eqnarray*}
 for Runge-Kutta 4,
 \begin{eqnarray*}
q=1,~~ a_1=1,~~t_{1,i} = t_{i},~~y_{1,i}={y}_{i},
\end{eqnarray*}
 for  Adams-Bashforth of order $1$,
 \begin{eqnarray*}
q=2,~~ a_1&=&\frac{3}{2},~~ a_2=-\frac{1}{2},\nonumber\\
t_{1,i} &=& t_{i},~~ y_{1,i}={y}_{i},~~ t_{2,i} = t_{i-1},y_{2,i}={y}_{i-1},
\end{eqnarray*}
 for  Adams-Bashforth of order $2$.

The final solution $y_N$ and the runtime (in seconds) of the mixed precision numerical method are denoted by $y_M$ and $T_M$, respectively. The final solution $y_N$ and the runtime of the SINGLE method (respectively, the DOUBLE method) are denoted by $y_S$ and $T_S$ (respectively, $y_D$ and $T_D$).

The runtime speedup of the SINGLE (or MIXED) method, measured relative to the runtime of the DOUBLE method, is defined as
$
\frac{T_D}{T_S}
\quad \text{(respectively, } \frac{T_D}{T_M} \text{)}.
$

\subsection{Benchmark models}\label{Sec:benchmarks}
In this section, we present two benchmark models that we will use to numerically show the efficiency of the mixed methods MAD(j), $0\le j\le q.$
\subsubsection{Benchmark model $1$}

For the first benchmark, we consider  a mathematical model for the regulation of the cell cycle by the circadian clock \cite{EBE2017}.  This ODE
 is of the form $(\ref{eqo1})$, defined on $[0,120]$, where  $ y=[y^{(1)},\ldots,y^{(d)}]^{T}, y^{(i)}=[y_{1}^{(i)},\ldots,y_{10}^{(i)}]$, $n=10d$,  and $y^{(i)}$ satisfies, the following system $(S_i)$, for $i=1,\ldots,d$
\begin{eqnarray}
\frac{dy_1^{(i)}}{dt} &=& 1/\tau \left ( \nu_{1b}(y_7^{(i)}+{\bf \Psi^{(i)}})/(k_{1b}(1+(y_3^{(i)}/k_{1i})^{p_0})+y_7^{(i)}+{\bf \Psi^{(i)}}) - k_{1d}y_1^{(i)} \right )\nonumber \\
\frac{dy_2^{(i)}}{dt} &=& 1/\tau \left ( k_{2b}{(y_1^{(i)})}^{q}-k_{2d}y_2^{(i)}-k_{2t}y_2^{(i)}+k_{3t}y_3^{(i)}  \right )\nonumber\\
\frac{dy_3^{(i)}}{dt} &=& 1/\tau \left  ( k_{2t}y_2^{(i)} - k_{3t}y_3^{(i)} - k_{3d}y_3^{(i)}  \right )\nonumber\\
\frac{dy_4^{(i)}}{dt} & =& 1/\tau \left ( \nu_{4b}(y_3^{r_0})^{(i)}/(k_{4b}^{r_0} +(y_3^{(i)})^{r_0}) - k_{4d}y_4^{(i)}  \right )\nonumber\\
\frac{dy_5^{(i)}}{dt} &=& 1/\tau\left ( k_{5b}y_4^{(i)} - k_{5d}y_5^{(i)} - k_{5t}y_5^{(i)} + k_{6t}y_6^{(i)}  \right )\nonumber\\
\frac{dy_6^{(i)}}{dt} &=& 1/\tau \left ( k_{5t}y_5^{(i)} - k_{6t}y_6^{(i)} - k_{6d}y_6^{(i)} + k_{7a}y_7^{(i)} - k_{6a}y_6^{(i)}  \right )\nonumber\\
\frac{dy_7^{(i)}}{dt} &=& 1/\tau \left  ( k_{6a}y_6^{(i)} - k_{7a}y_7^{(i)} - k_{7d}y_7^{(i)}  \right )\nonumber\\
\frac{dy_8^{(i)}}{dt} &=&  \lambda \left  ( (k_{_{impf}}+k_{_{0mpf}}\exp(-\eta d))k_{_{1mpf}}^{n_0}/( k_{_{1mpf}}^{n_0} + (y_8^{(i)})^{n_0} + s (y_{10}^{(i)})^{n_0} ) (1-y_8^{(i)}) - d_{wee1}y_9^{(i)}y_8^{(i)}  \right )\nonumber\\
\frac{dy_9^{(i)}}{dt} &=& \lambda \left ( \frac{k_{_{actw}}}{k_{_{actw}}+d_{w1}}(c_w+C(y_7 - b_{bmal0}) + b_{bmal0} ) +... \right .\nonumber\\
        & & \left . (\frac{k_{_{actw}}}{k_{_{actw}}+d_{w1}}-1)k_{inactw}{(y_8^{(i)})^{n_0}}y_9^{(i)}/(k_{1wee1}^{n_0}+{(y_8^{(i)})^{n_0}})  - d_{w2}y_9^{(i)} \right )\nonumber\\
\frac{dy_{10}^{(i)}}{dt} &=& \lambda  \left  ( k_{_{act}}(y_8^{(i)}-y_{10}^{(i)})  \right ).\nonumber
\end{eqnarray}
The systems $(S_i)$, $i=1,\ldots,d,$ are strongly coupled with each other via the variable ${\bf \Psi^{(i)}}$, that is given by
$$
\begin{aligned}
{\bf \Psi^{(i)}}&=\frac{k_{s}}{d}\displaystyle \sum_{j=1}^{d}\arctan(y_2^{(j)}-y_2^{(i)})+\frac{\pi}{2}k_s,\ i=1,\ldots,d.
\end{aligned}
$$
The constant parameters used in this model, along with the initial solution $y_0$, are provided in Appendix $1$. The size $n$ of this ODE, will be defined in the numerical tests.

\subsubsection{Benchmark model $2$}
For the second benchmark, we consider a neural field model that treats the
cortex as a continuous space and describes the spatiotemporal dynamics of the neural activity given long range interactions \cite{A1977,WC1972}. The neuronal membrane potential $V(x,t)$
at location $x$ and time $t$ follows the integral-differential equation
\begin{eqnarray}\label{PDE:NFE}
  \frac{\partial{V}}{{\partial t}}(x,t) &=&I(x,t)-V(x,t)+\int_{-1}^{1}K(|x-z|)S(V(z,t)) dz,\qquad x\in [-1,1], t\in [0,1],
\end{eqnarray}
where $I(x,t)$ is an external source of stimulus; $S(V)$ is the dependence between the firing rate of the neuron and the membrane potential and $K(|x-z|)$ is  the connectivity between neurons at locations $x$ and $z$.

The integral term can be discretized. Let  $x_i=-1+i\Delta x, i = 0,\ldots, d,$ be a uniform grid of the interval $[-1,1].$ Then
$$
\int_{-1}^{1}K(|x-z|)S(V(z,t)) dz=\displaystyle \sum_{i=0}^{d-1}\int_{x_{i}}^{x_{i+1}}K(|x-z|)S(V(z,t)) dz.
$$
In each subinterval $[x_i,x_{i+1}],$ we introduce $k_{ \ell}$
Gaussian nodes: $x_{i,s}=x_{i}+\frac{\Delta x}{2}(1+\tau_{s}), s = 1, ..., k_{ \ell}$, where $\tau_{s}$ are the
roots, sorted in an increasing order, of the $k_{\ell}$-th degree Legendre polynomial.  Using the composite left rectangle method to evaluate the integral $\int_{x_{i}}^{x_{i+1}}K(|x-z|)S(V(z,t)) dz$  with the points
 $\{x_{i,s}\}_{1\le s \le k_{\ell}}$, and then evaluating (\ref{PDE:NFE}) at these points, we get a  system of ODEs of the form $(\ref{eqo1}),$ where $n=d\times k_{\ell}, y(t)=[y_1(t),\ldots,y_{d}(t)]^{T}$ with $$y_{i}(t)=[y_1^{(i)}(t),\ldots,y_{k_{\ell}}^{(i)}(t)]\approx[V(x_{i,1},t),\ldots,V(x_{i,k_{ \ell}},t)]\in \mathbb{R}^{k_{\ell}}.$$
The parameters used in our numerical tests are
  $K(x)=e^{-x^2}, S(x)=I(x,t)=\tanh(x)$, and $k_{\ell}=10.$
The initial solution $y_0$ is given by $[y_{1}^{(0)},\ldots,y_{n}^{(0)}]^{T}$ where $y_{i}^{(0)}=\exp(6(i-0.5n)/n),i=1,\ldots,n.$ The size $n$ of this ODE, will be defined in the numerical tests.

\subsection{Computational environment}
All the numerical tests were launched on the computer cluster Gros on Grid5000\footnote{\url{https://www.grid5000.fr/}} (Intel Xeon Gold 5220, 2.20GHz) on Ethernet (25G) a network. The DOUBLE, SINGLE, MAS(j) and MAD(j)  numerical methods, given in section \ref{Sec:tests}, were implemented in Fortran in parallel with MPI and compiled on GCC 10.4.0 with the optimization option O3. Each evaluation of the function $f(t,y)$ is parallelized using MPI (version 4.1.5), and the numerical tests were launched with  $14$ hosts ($252$ processors).
Some numerical tests, when the size of the ODE is small, will be launched in sequential to show that the parallel computing with MPI does not impact the quality of our mixed methods.  For each numerical test, we ran $4$ technical replicates of the simulations. Runtimes (in seconds) were
averaged. The relative error was computed as the infinite norm of coefficient-wise relative errors $\left \|\frac{y_{_{D}}-y_{_{T}}}{y_{_{D}}}\right \|_{\infty},$ where  $\frac{\mbox{ \ \ \ \    }}{\mbox{    }}$ is the pointwise division of two vectors,  with $y_{_{T}}\in \{y_{_{S}},y_{_{M}}\}$.

\section{Results}\label{Sec:results}
This section shows theoretically and numerically  the efficiency of our mixed methods MAD(j), $j\le q.$

\subsection{Theoretical results}
The following theorem shows that our mixed methods $\mathrm{MAD}(j)$, for $j \le q$, provide a very good approximation of the double solution, unlike the SINGLE and MAS methods. At the same time, MAD(j) accelerates computations without compromising the integrity of the numerical method (\ref{E_S1}).
   \begin{theorem}\label{Theorem 1}

Let us assume that the function $f(t,y)$ is Lipschitz with Lipschitz's constant $L.$ We have the following results:

   \begin{enumerate}
   \item   Let $\tilde{y}_i$ be the solution computed by  the method $\mathrm{MAD}(j_0)$, ${j_0}\le q.$ Then we have :
    \begin{eqnarray*}
 \displaystyle\max_{0\le i \le N}\|{y}_{_{i}}-\tilde{y}_{_{i}}\| \le c_{1,j_0} \varepsilon+c_{2,j_0} \varepsilon \dt^{p}
 \end{eqnarray*}
 where $c_{1,j_0},c_{2,j_0}$ are real constants and $p$ is the order of the numerical method.
 \item  Let $\tilde{y}_i$ be the solution computed by  the method $\mathrm{MAS}(j_0)$, ${j_0}\le q.$ Then we have :
  \begin{eqnarray*}
  \displaystyle\max_{0\le i \le N}\|{y}_{_{i}}-\tilde{y}_{_{i}}\| \le c_{1,j_0} \varepsilon+c_{2,j_0} \varepsilon \dt^{p}+c_3 {\frac{\varepsilon}{\dt}}.
 \end{eqnarray*}
  where $c_{1,j_0},c_{2,j_0},  c_3$ are real constants.
    \end{enumerate}
 \end{theorem}
   \begin{pf}
   \begin{enumerate}
\item

Since $f(t,y)$ is Lipschitz with constant $L$, and using the property of the model $\fl$, we get
$$
\begin{aligned}
\|f^{(L)}(t,z)-f(t,z)\|&=\| \fl(f(\fl(t),\fl(z)))-f(t,z)\|\\
&\le \| \fl(f(\fl(t),\fl(z)))-f(\fl(t),\fl(z))\| +\|f(\fl(t),\fl(z))-f(t,z)\|\\
&\le \varepsilon \|f(\fl(t),\fl(z))\|+L \left [ |t-\fl(t)|+\|z-\fl(z)\| \right ]\\
& \le  \varepsilon \|f(\fl(t),\fl(z))\|+\varepsilon L\left [ |t|+\|z\| \right ]\\
& \le  \varepsilon \left [\|f\|_{\infty}+L (|t|+\|z\|) \right ]\qquad \mbox{ for all }(t,z).
\end{aligned}
$$

Thus,
\begin{equation}\label{Eq_K}
\|\tilde{k}_{l,i}-{k}_{l,i}\|=\|f^{(L)}(t_{l,i},y_{l,i})-f(t_{l,i},y_{l,i})\| \le \varepsilon \left [\|f\|_{\infty}+L (|t_{l,i}|+\|y_{l,i}\|) \right ]\quad \forall i.
\end{equation}

Using the definition of $y_{l,i}$ given by $(\ref{E_y_li})$, we get :
$$
\begin{aligned}
\|y_{l,i}\|&\le \sum_{j=0}^{q} |\alpha_{l,j}| \|y_{i-j}\|+h\sum_{j=1}^{l-1} |\beta_{l,j}| \|k_{j,i}\|.
\end{aligned}
$$

Thus, as $$\|y_i\|=\|y_i-y(t_i)\|+\|y(t_i)\|\le C_y \dt^p+\|y\|_{\infty} \mbox{ and } \|k_{j,i}\| \le \|f\|_{\infty},$$ we get
\begin{equation}\label{EQ__Y_lj}
\|y_{l,i}\|\le \left [ C_y\sum_{j=0}^{q} |\alpha_{l,j}| \right  ] \dt^p+ \left [ \sum_{j=0}^{q} |\alpha_{l,j}| \right ] \|y\|_{\infty}+h\left [\sum_{j=1}^{l-1} |\beta_{l,j}|\right] \|f\|_{\infty}.
\end{equation}

Using the previous inequality, we get from (\ref{Eq_K})

$$
\begin{aligned}
\|\tilde{k}_{l,i}-{k}_{l,i}\| &\le \varepsilon\left [ (1 +L h \sum_{j=1}^{l-1} |\beta_{l,j}|)\|f\|_{\infty}+L T+L \|y\|_{\infty} \sum_{j=0}^{q} |\alpha_{l,j}|  \right ] +\left[ L C_y\sum_{j=0}^{q} |\alpha_{l,j}|\right ]\varepsilon \dt^p.
\end{aligned}
$$
Therefore
\begin{equation}\label{IN_Kl}
\|\tilde{k}_{l,i}-{k}_{l,i}\| \le K_{1,l}^{(h)}\varepsilon +K_{2,l}\varepsilon \dt^p,
\end{equation}
where
$$
\begin{aligned}
K_{1,l}^{(h)}&=(1 +L h \sum_{j=1}^{q-1} |\beta_{l,j}|)\|f\|_{\infty}+L T+L \|y\|_{\infty} \sum_{j=0}^{q} |\alpha_{l,j}|\le K_1,  \\
K_{2,l}&=L C_y\sum_{j=0}^{q} |\alpha_{l,j}|\le K_2,
\end{aligned}
$$
with
$$
\begin{aligned}
 K_1&=(1 +L (T-t_0) \displaystyle \sum_{j=1}^{q-1}{\displaystyle \max_{1\le l\le q} |\beta_{l,j}|})\|f\|_{\infty}+L T+L \|y\|_{\infty} {\displaystyle \max_{1\le l\le q} \sum_{j=0}^{q} |\alpha_{l,j}|}, \\
K_2&={\displaystyle \max_{1\le l\le q} K_{2,l}}
\end{aligned}
$$
Then, from (\ref{E_S1_L}) we have:
  \begin{eqnarray*}
 \qquad\tilde{y}_{i+1}&={\displaystyle \underbrace{{y}_{i}+h\displaystyle \sum_{l=1}^{q}a_{l}k_{l,i}}_{=y_{i+1}}}-({y}_{i}-\tilde{y}_{i})+h \displaystyle \sum_{\underset {{l\in \{l_{1},\ldots,l_{q-{j_0}}\} }}{l=1}}^{q}a_{l}(\tilde{k}_{l,i}-k_{l,i}),
 \end{eqnarray*}
 and therefore, we get
 $$
 \begin{aligned}
 \|y_{i+1}- \tilde{y}_{i+1}\|\le \|y_{i}- \tilde{y}_{i}\|+\dt \displaystyle \sum_{\underset {{l\in \{l_{1},\ldots,l_{q-{j_0}}\} }}{l=1}}^{q}|a_{l}| \|\tilde{k}_{l,i}-k_{l,i}\|.
 \end{aligned}
 $$
By induction, and since $y_0=\tilde{y}_0$, we get
 $$
 \begin{aligned}
 \|y_{i+1}- \tilde{y}_{i+1}\|\le \dt \displaystyle \sum_{m=0}^{i} \left ( \displaystyle \sum_{\underset {{l\in \{l_{1},\ldots,l_{q-j_0}\} }}{l=1}}^{q}|a_{l}| \|\tilde{k}_{l,m}-k_{l,m}\|\right ).
 \end{aligned}
 $$
 Using (\ref{IN_Kl}), we obtain
$$
 \begin{aligned}
 \|y_{i+1}- \tilde{y}_{i+1}\|&\le (i+1) \dt C_a  (K_1\varepsilon +K_2\varepsilon \dt^p)
 \end{aligned}
 $$
where $C_a= \displaystyle \sum_{\underset {{l\in \{l_{1},\ldots,l_{q-{j_0}}\} }}{l=1}}^{q}|a_{l}|.$  Since $(i+1)\dt \le Nh=T-t_0$, we get
 $$
 \begin{aligned}
 \|y_{i+1}- \tilde{y}_{i+1}\|&\le  (T-t_0) K_1 C_a\varepsilon +(T-t_0) K_2 C_a \varepsilon \dt^p\\
 &\le c_{1,j_0} \varepsilon +c_{2,j_0}\varepsilon \dt^p.
 \end{aligned}
 $$
 where
 $$
\begin{aligned}
c_{1,j_0}&=(T-t_0) K_1 C_a,  \ c_{2,j_0}&=(T-t_0) K_2C_a.
\end{aligned}
$$
The constant $K_1$, in the above inequality, could be replaced by $K_{1,l}^h.$ Since $h\mapsto K_1^h$ is increasing then as $h$ decreases $ K_{1,l}^h$ decreases, and therefore the error between the mixed solution and the double solution decreases.

\item
In this case we have :
$$
\tilde{y}_{i+1}=\fl(\tilde{y}_{i})+h \displaystyle \sum_{\underset {{l\notin \{l_{1},\ldots,l_{q-{j_0}}\} }}{l=1}}^{q}a_{l} {{k}_{l,i}}+h\displaystyle \sum_{\underset {{l\in \{l_{1},\ldots,l_{q-{j_0}}\} }}{l=1}}^{q}a_{l}\tilde{k}_{l,i}.
$$
Therefore, using the same technique as the previous part, we get
\begin{eqnarray*}
 \qquad\tilde{y}_{i+1}&={\displaystyle \underbrace{{y}_{i}+h\displaystyle \sum_{l=1}^{q}a_{l}k_{l,i}}_{=y_{i+1}}}-({y}_{i}-\fl(\tilde{y}_{i}))+h \displaystyle \sum_{\underset {{l\in \{l_{1},\ldots,l_{q-{j_0}}\} }}{l=1}}^{q}a_{l}(\tilde{k}_{l,i}-k_{l,i}),
 \end{eqnarray*}
As
$$
\|{y}_{i}-\fl(\tilde{y}_{i})\| \le \|{y}_{i}-\tilde{y}_{i}\|+\varepsilon \|\tilde{y}_i\|,
$$
then we get
$$
 \begin{aligned}
 \|y_{i+1}- \tilde{y}_{i+1}\|\le \|y_{i}- \tilde{y}_{i}\|+\dt \displaystyle \sum_{\underset {{l\in \{l_{1},\ldots,l_{q-{j_0}}\} }}{l=1}}^{q}|a_{l}| \|\tilde{k}_{l,i}-k_{l,i}\|+{\varepsilon \|\tilde{y}_i\|}.
 \end{aligned}
 $$
By induction, we get
 $$
 \begin{aligned}
 \|y_{i+1}- \tilde{y}_{i+1}\|\le \dt \displaystyle \sum_{m=0}^{i} \left ( \displaystyle \sum_{\underset {{l\in \{l_{1},\ldots,l_{q-{j_0}}\} }}{l=1}}^{q}|a_{l}| \|\tilde{k}_{l,m}-k_{l,m}\|\right )+\varepsilon \displaystyle \sum_{m=0}^{i}{\|\tilde{y}_m\|}.
 \end{aligned}
 $$

As
$$
 \begin{aligned}
\varepsilon \displaystyle \sum_{m=0}^{i}{\|\tilde{y}_m\| } &\le (i+1)\varepsilon \max_{0\le m\le N}\|\tilde{y}_m\| \\
&\le {N\varepsilon}{\max_{0\le m\le N}\|\tilde{y}_m\| }\\
&\le  {\frac{T-t_0}{\dt}\varepsilon}{\max_{0\le m\le N}\|\tilde{y}_m\|}
 \end{aligned}
$$

Using the previous inequality and the part $1,$ we get
$$
 \begin{aligned}
 \displaystyle\max_{0\le i \le N}\|{y}_{_{i}}-\tilde{y}_{_{i}}\| \le c_{1,j_0} \varepsilon +c_{2,j_0} \varepsilon \dt^p +c_3 {\frac{\varepsilon}{\dt}},
 \end{aligned}
$$
where the constant $c_3$ is given by
$$(T-t_0){\displaystyle  \max_{0\le m\le N}\|\tilde{y}_m\|}.~\qquad ~\blacksquare$$
\end{enumerate}
   \end{pf}
   \begin{rmk}\label{Remark_1}
   This theorem emphasizes the importance of performing the accumulation in equation~(\ref{E_S1_L}) in double precision (as in $\MIXEDDjD$). If the accumulation is instead carried out in single precision (as in $\MIXEDSINGLE$), even when all other computations are performed in double precision, the error between the double precision and mixed precision solutions is of order $\frac{\varepsilon}{h}$ rather than $\varepsilon$.

Consequently, as the time step $h$ decreases and tends to zero, the error between the double precision solution and each solution produced by $\MIXEDDjD$ decreases, whereas the error between the double precision solution and each solution produced by $\MIXEDSINGLE$ solution increases without bound. This clearly demonstrates that the $\MIXEDDjD$ methods provide superior accuracy compared with the $\MIXEDSINGLE$ methods.

Furthermore, the constants $c_{1,j}$ and $c_{2,j}$ appearing in the theorem involve the sum $\displaystyle\sum_{\underset{l \in \{l_{1},\ldots,l_{q-j}\}}{l=1}}^{q} |a_l|.$ Reducing the number of stages  $j$  computed in double precision increases the value of this sum, and hence enlarges the constants $c_{1,j}$ and $c_{2,j}$, thereby slightly degrading the accuracy of the mixed precision solution.

   \end{rmk}
   \subsection{Numerical results}
   This section presents a numerical assessment of the efficiency of the mixed precision methods MAD(j), $j=0,\ldots,q$, applied to  RK4, RK2, AB1, and AB2, running in both sequential and parallel, for solving the benchmark problems introduced in Section~\ref{Sec:benchmarks}.

\subsubsection{Numerical results of parallel numerical methods}
This section shows the numerical results of the parallel numerical methods.\\

{\bf Benchmark $1$} \\

We present an analysis of the numerical results obtained for Benchmark~1.
The results corresponding to the timesteps $12 \times 10^{-4}$,
$4 \times 10^{-4}$, $12 \times 10^{-5}$, and $12 \times 10^{-6}$
are displayed in
Figures~\ref{BEN1__TS_12M4__WPD_DOUBLE},
\ref{BEN1__TS_4M4__WPD_DOUBLE},
\ref{BEN1__TS_12M5__WPD_DOUBLE}, and
\ref{BEN1__TS_12M6__WPD_DOUBLE}, respectively.

For the smallest timestep $12 \times 10^{-6}$, the computational cost becomes
significant; therefore, the number of processors was increased to~1008.
With this timestep, we only present the numerical results for AB1 and AB2 methods, as these methods are less expensive than the others since
they only require one RHS evaluation at each integration step.

In all Figures \ref{BEN1__TS_12M4__WPD_DOUBLE}, \ref{BEN1__TS_4M4__WPD_DOUBLE}, \ref{BEN1__TS_12M5__WPD_DOUBLE}, \ref{BEN1__TS_12M6__WPD_DOUBLE}, \ref{BEN2__WPD_DOUBLE}, \ref{BEN1__SEQ_DOUBLE}, the horizontal axis represents the runtime speedup measured
with respect to the fully double precision corresponding numerical method, while the vertical
axis shows the relative error computed with respect to the double precision
 solution. This representation provides a direct visualization of the
trade-off between computational efficiency and numerical accuracy.

First, we are going to explain the numerical results for RK4. For the timestep $12 \times 10^{-4}$ (Figure~\ref{BEN1__TS_12M4__WPD_DOUBLE}),
the solution computed by each mixed method  MAD($j$) is a good approximation of the solution of the DOUBLE method.
The best mixed method is MAD(0), as it achieves a runtime comparable
to that of the fully SINGLE method and is up to $2.2$ times faster
than the fully DOUBLE method, while maintaining high accuracy.
In contrast, the fully SINGLE method exhibits a significant loss of
accuracy. The relative error associated with the MAD($j$) solutions remains
of order $10^{-6}$, whereas the MAS variants yield errors of order
$5 \times 10^{-3}$. This substantial gap highlights the importance of
performing the accumulation step in double precision.

For the timesteps $4 \times 10^{-4}$ and $12 \times 10^{-5}$ (Figures~\ref{BEN1__TS_4M4__WPD_DOUBLE} and \ref{BEN1__TS_12M5__WPD_DOUBLE}),
the relative error of the MAD($j$) solutions slightly decreases,
while the errors of both the fully SINGLE and MAS solutions increase.
These observations are consistent with the theoretical results established
in Theorem~1. When we reduce more the timestep to  $12 \times 10^{-6}$ (Figure~\ref{BEN1__TS_12M6__WPD_DOUBLE}), the relative error of the MAD($j$)
solutions is of order $10^{-6}$, whereas the error of the  SINGLE and MAS solutions is approximately $2.5$. This confirms that further reduction of the timestep
leads to the behavior predicted by Theorem~1, that is the divergence of the SINGLE solution from the DOUBLE solution. However, the MAD(0) mixed precision method and the SINGLE method
achieve essentially identical runtime performance, both being up to
$2.5$ times faster than the fully DOUBLE method, while their accuracy differs markedly.

We have the same conclusions for the numerical methods AB1, AB2 and RK2.
 In all numerical methods, the SINGLE and MAS methods of RK4, RK2, AB1, and AB2 provide poor approximations of the
double precision solution. This loss of accuracy results from performing
the accumulation of the numerical method in single precision, causing rounding errors
to accumulate at each timestep.

To further illustrate the efficiency of our mixed methods MAD, even for very small time steps, Table \ref{RES_Euler_Small} reports numerical results obtained when the AB1 method is used for solving Benchmark~1 on the interval $[0,T_f]$, where $T_f$ is taken to be very small. Let $T_{M_1}$ and $y_{M_1}$ (respectively $T_{M_2}$ and $y_{M_2}$) be the runtime and the solution of the mixed method $MAD(0)$ (respectively $MAS(0)$). Table \ref{RES_Euler_Small}  clearly shows that, as the time step $h$ decreases, the relative error of the solution $y_{M_1}$ decreases. This indicates that, even for very small time steps, this mixed solution remains a very good approximation of the double solution and that our mixed method $MAD(0)$ is up to $2.5$ times faster than the
DOUBLE method.

\begin{table}[htbp]
\begin{center}
\begin{tabular}{|c|c|c|c|c|c|c|c|}
 \hline
 $\dt$ &$T_f$&$\frac{T_D}{T_S}$&$\frac{T_D}{T_{M_1}}$&$\frac{T_D}{T_{M_2}}$&$\left \|\frac{y_{_{D}}-y_{_{M_1}}}{y_{_{D}}}\right \|_{\infty}$&$\left \|\frac{y_{_{D}}-y_{_{M_2}}}{y_{_{D}}}\right \|_{\infty}$&$\left \|\frac{y_{_{D}}-y_{_{S}}}{y_{_{D}}}\right \|_{\infty}$ \\ \hline
 $10^{-6}$ &$1$&$2.6$&$2.2$&$0.99$&$5.1\times 10^{-9}$&$7.9\times 10^{-2}$&$7.9\times 10^{-2}$ \\ \hline
 $10^{-8}$ &$10^{-2}$&2.5&$2.2$&$0.93$&$2.6\times 10^{-9}$&$5.5\times 10^{-2}$ &$5.5\times 10^{-2}$ \\ \hline
  $3.3\times 10^{-9}$ &$10^{-2}$&$2.6$&$2.3$&$0.95$&$2.6\times 10^{-9}$&$1.5\times 10^{-1}$&$1.5\times 10^{-1}$ \\ \hline
 $10^{-9}$&$10^{-4}$&$2.5$&$2.2$&$0.99$&$1.6\times 10^{-9}$&$5.8\times 10^{-3}$&$5.8\times 10^{-3}$\\ \hline
  $10^{-10}$&$10^{-4}$&$2.5$&$2.4$&$0.99$&$1.6\times 10^{-9}$&$5.5\times 10^{-2}$&$5.5\times 10^{-2}$\\ \hline
\end{tabular}
\end{center}
 \caption{Explicit Euler method} \label{RES_Euler_Small}
 \end{table}

The above numerical results show that performing the accumulation in double precision is essential to preserve numerical stability and accuracy. These observations are fully consistent with Theorem~1, which predicts the loss of convergence when the entire computation is carried out in single precision.

Overall, the mixed precision strategy successfully accelerates
double precision computations without compromising numerical reliability.
In particular, the solution of the  MAD($j$) method maintains the quality of the
double precision solution independently of the time integration method.
Among all tested configurations, MAD(0) provides the best balance
between runtime and accuracy: all stage evaluations are performed in
single precision, whereas the solution accumulation is carried out in
double precision, thereby reducing computational cost while preserving
stability and accuracy.

\begin{figure}[htbp]
\centering
\includegraphics[scale=0.5]{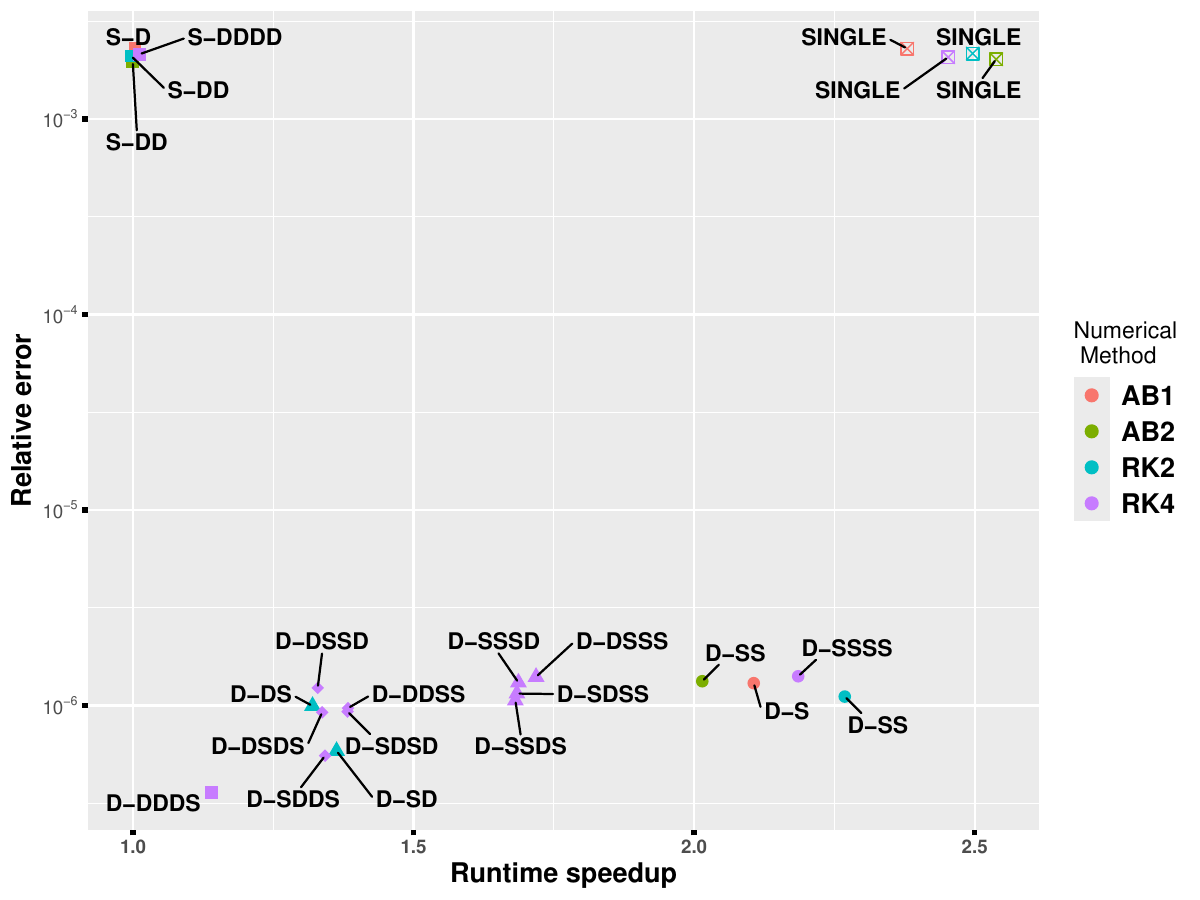}
\caption{"Work Precision Diagrams" (WPD) for benchmark $1$, running in parallel with MPI, with the timestep $12\times 10^{-4}.$ The runtime speedup against accuracy.}\label{BEN1__TS_12M4__WPD_DOUBLE}
\end{figure}

\begin{figure}[htbp]
\centering
\includegraphics[scale=0.5]{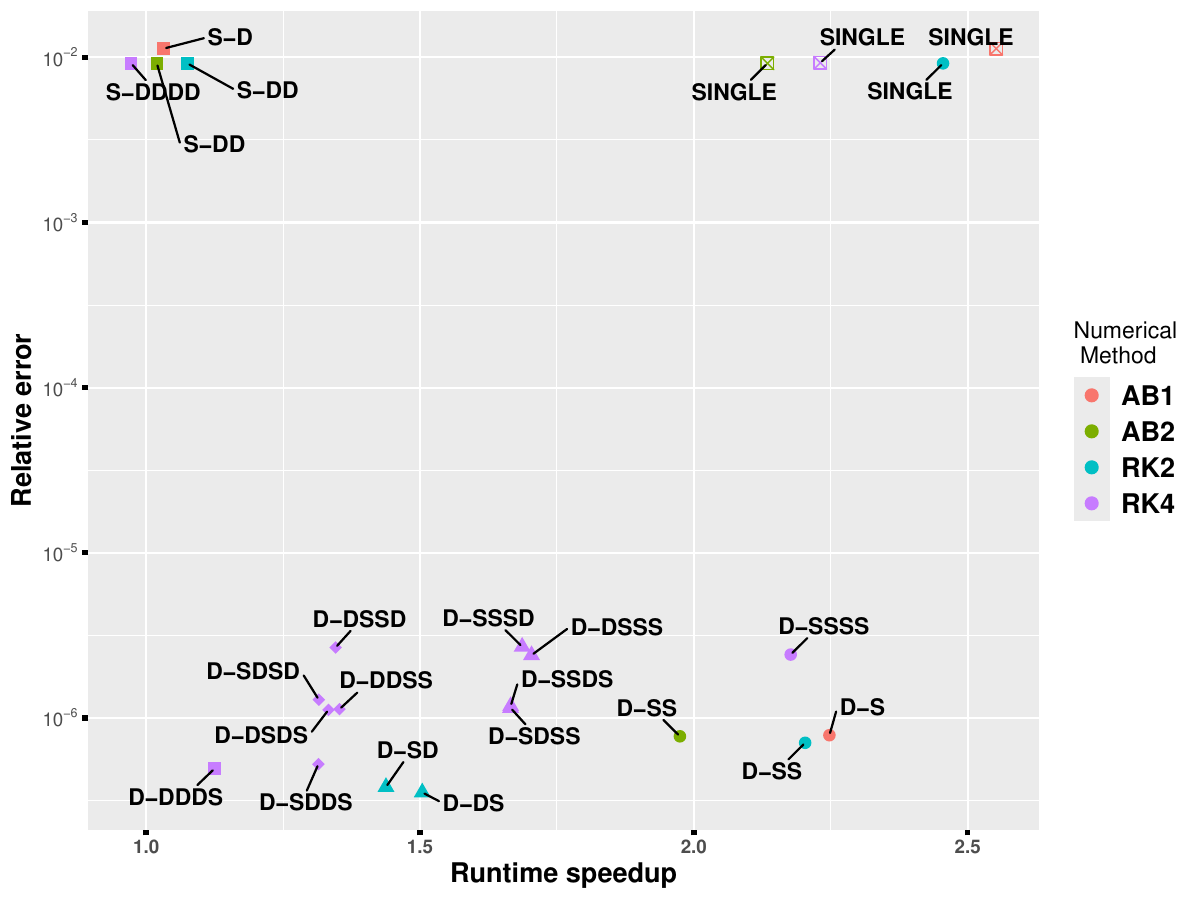}
\caption{"Work Precision Diagrams" (WPD) for benchmark $1$, running in parallel with MPI, with the timestep $4\times 10^{-4}.$ The runtime speedup against accuracy.}\label{BEN1__TS_4M4__WPD_DOUBLE}
\end{figure}

\begin{figure}[htbp]
\centering
\includegraphics[scale=0.5]{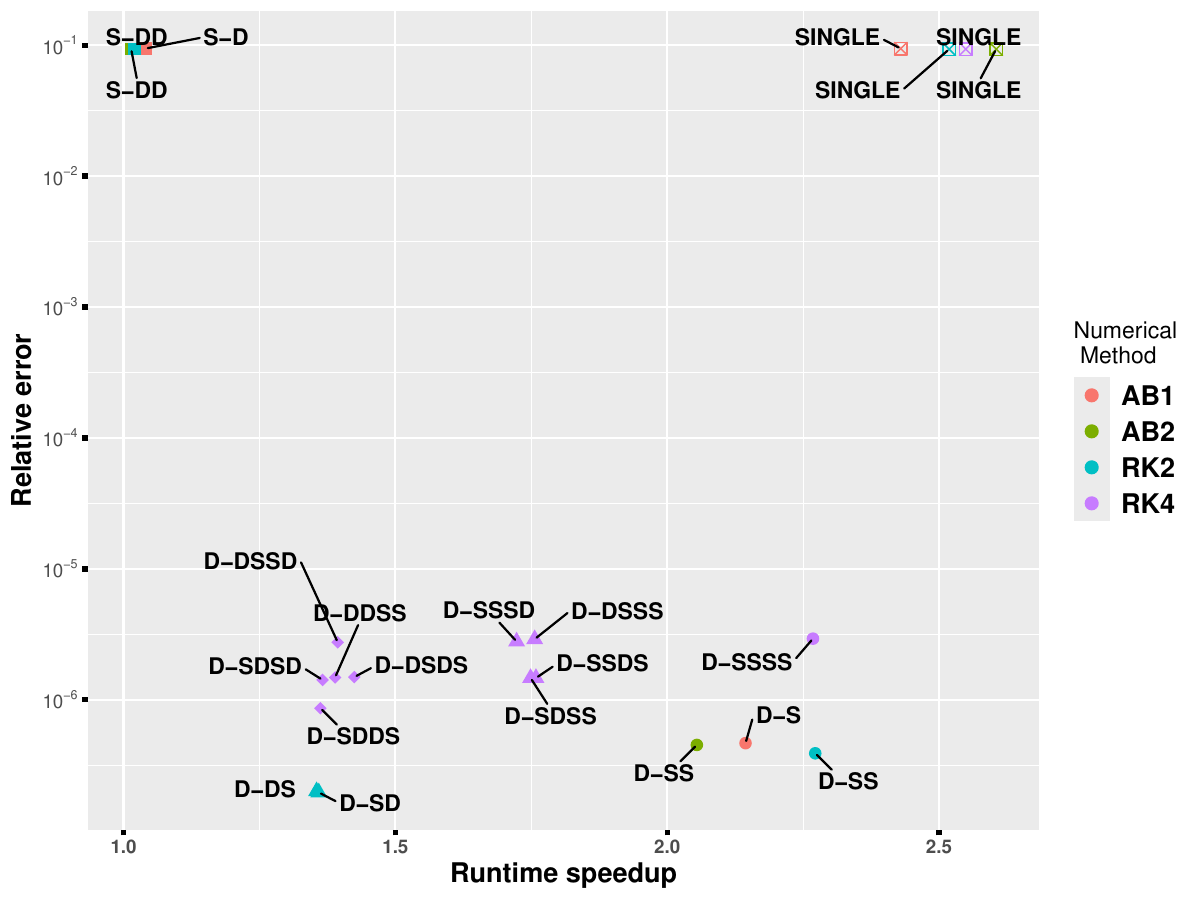}
\caption{"Work Precision Diagrams" (WPD) for benchmark $1$, running in parallel with MPI, with the timestep $12\times 10^{-5}.$ The runtime speedup against accuracy.}\label{BEN1__TS_12M5__WPD_DOUBLE}
\end{figure}

\begin{figure}[htbp]
\centering
\includegraphics[scale=0.5]{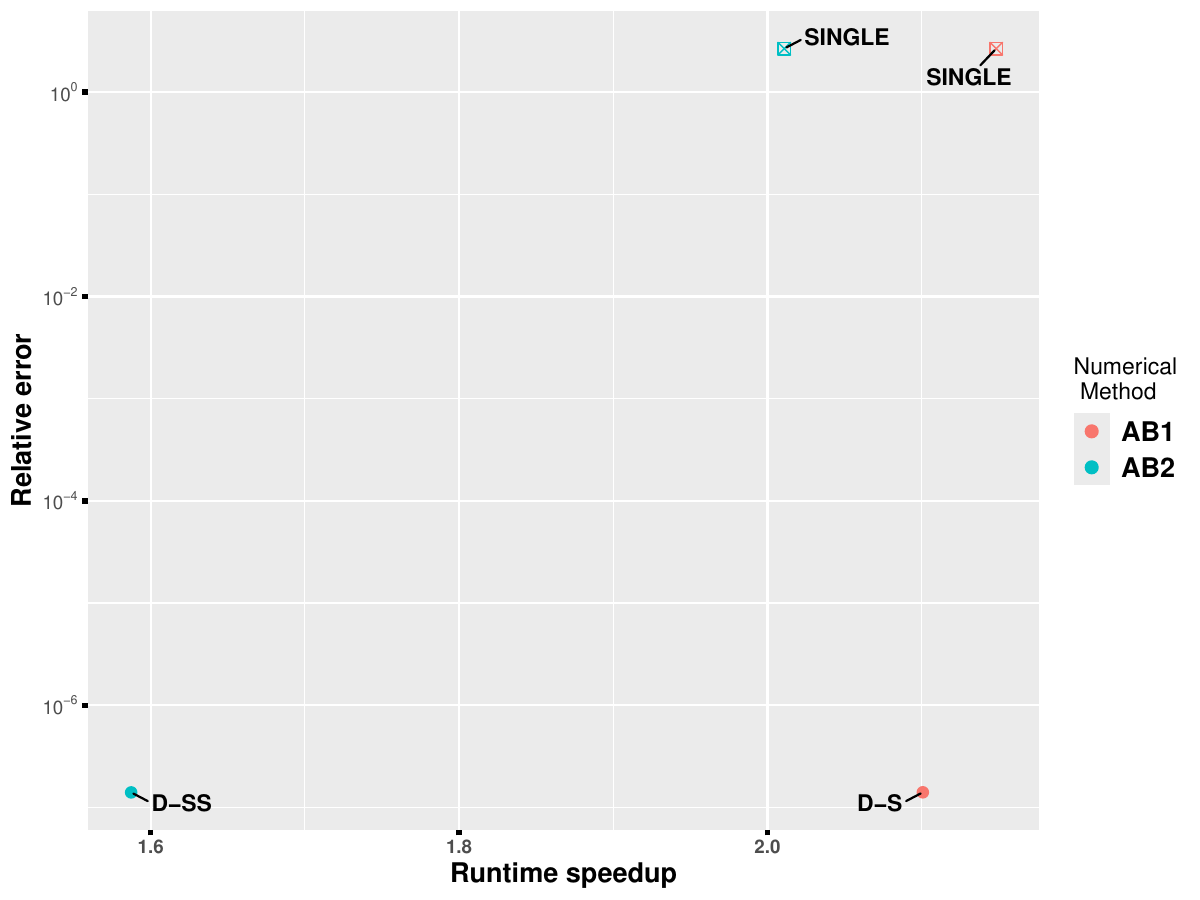}
\caption{"Work Precision Diagrams" (WPD) for benchmark $1$, running in parallel with MPI, with the timestep $12\times 10^{-6}.$ The runtime speedup against accuracy.}\label{BEN1__TS_12M6__WPD_DOUBLE}
\end{figure}

{\bf Benchmark $2$} \\

The numerical results obtained for Benchmark~2 are reported in
Figure~\ref{BEN2__WPD_DOUBLE}. Similarly to Benchmark~1, this figure
illustrates the trade-off between runtime speedup and the relative error.

We observe that all mixed precision variants, as well as the fully
SINGLE implementation, achieve speedups of up to a factor of $2$
compared to the DOUBLE method. However, the accuracy behavior differs
significantly between the mixed strategies. In particular, the relative
errors associated with the MAD($j$) methods remain considerably smaller
than those obtained with the MAS method. This confirms that performing
the accumulation step in double precision effectively controls the
propagation of rounding errors, even when stage evaluations are carried
out in single precision.

As in Benchmark~1, the MAS method leads to noticeably larger errors
due to the full single precision accumulation of the numerical method,
which amplifies rounding effects over successive timesteps. In contrast,
the MAD($j$) methods preserve a level of accuracy comparable to the
DOUBLE method while still providing substantial performance gains.

Overall, the numerical experiments for Benchmark~2 reinforce the
conclusions drawn from Benchmark~1: the MAD(0) configuration offers the
best compromise between computational efficiency and numerical accuracy.
By computing all stage evaluations in single precision and accumulating
the solution in double precision, MAD(0) achieves significant runtime
acceleration without degrading the quality of the DOUBLE solution.

\begin{figure}[htbp]
\centering
\includegraphics[scale=0.5]{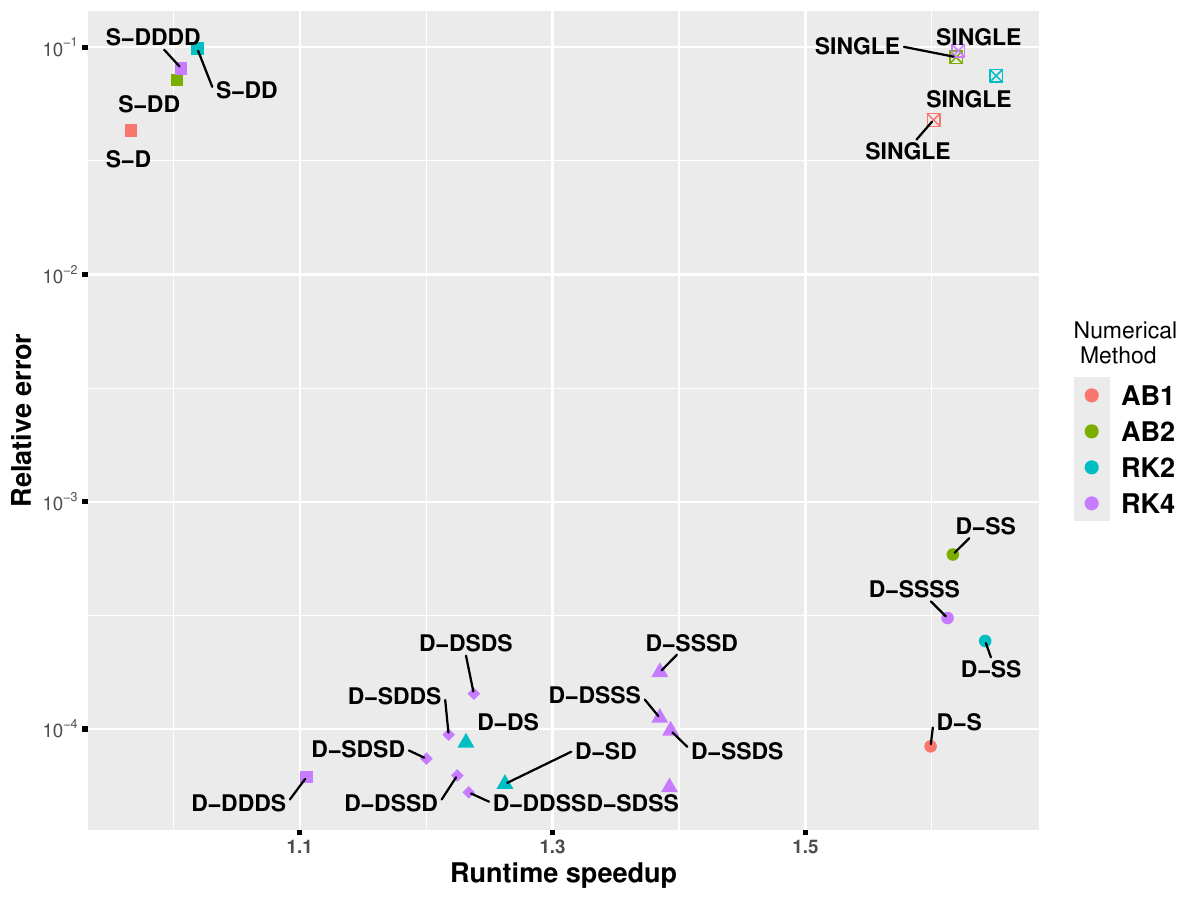}
\caption{"Work Precision Diagrams" (WPD) for benchmark $2$, running in parallel with MPI,  with the timestep $h=10^{-3}$. The runtime speedup against "double" accuracy.}\label{BEN2__WPD_DOUBLE}
\end{figure}

\subsubsection{Numerical results of sequential numerical methods}
In order to assess whether the MPI-based parallelization affects the numerical quality of the proposed mixed methods.
Figure~\ref{BEN1__SEQ_DOUBLE} reports the numerical results obtained for Benchmark~1, when the numerical methods are coded in sequential mode, with the timestep of $12 \times 10^{-4}$.

Because the code is executed sequentially, the problem size was reduced to $10^{4}$. For the original size of $10^{5}$,
the runtime, in sequential mode, becomes prohibitively large and does not allow for a practical comparison. This reduction ensures that the
sequential experiments remain computationally feasible while preserving the qualitative behavior of the numerical solutions.

The results displayed in Figure~\ref{BEN1__SEQ_DOUBLE} clearly indicate  that the mixed precision methods achieve speedups of up to $2.2$
relative to the fully double precision implementation, while maintaining a high level of accuracy. The relative error remains consistent with that observed in the parallel experiments.

These observations show that the MPI parallelization does not alter the numerical behavior of the mixed precision approach.
In particular, the performance gains obtained with the MAD($j$) method stem from the arithmetic precision design rather than from
parallel effects. This confirms that the effectiveness of the proposed mixed precision methodology is independent of the execution mode,
whether sequential or parallel.

\begin{figure}[htbp]
\centering
\includegraphics[scale=0.5]{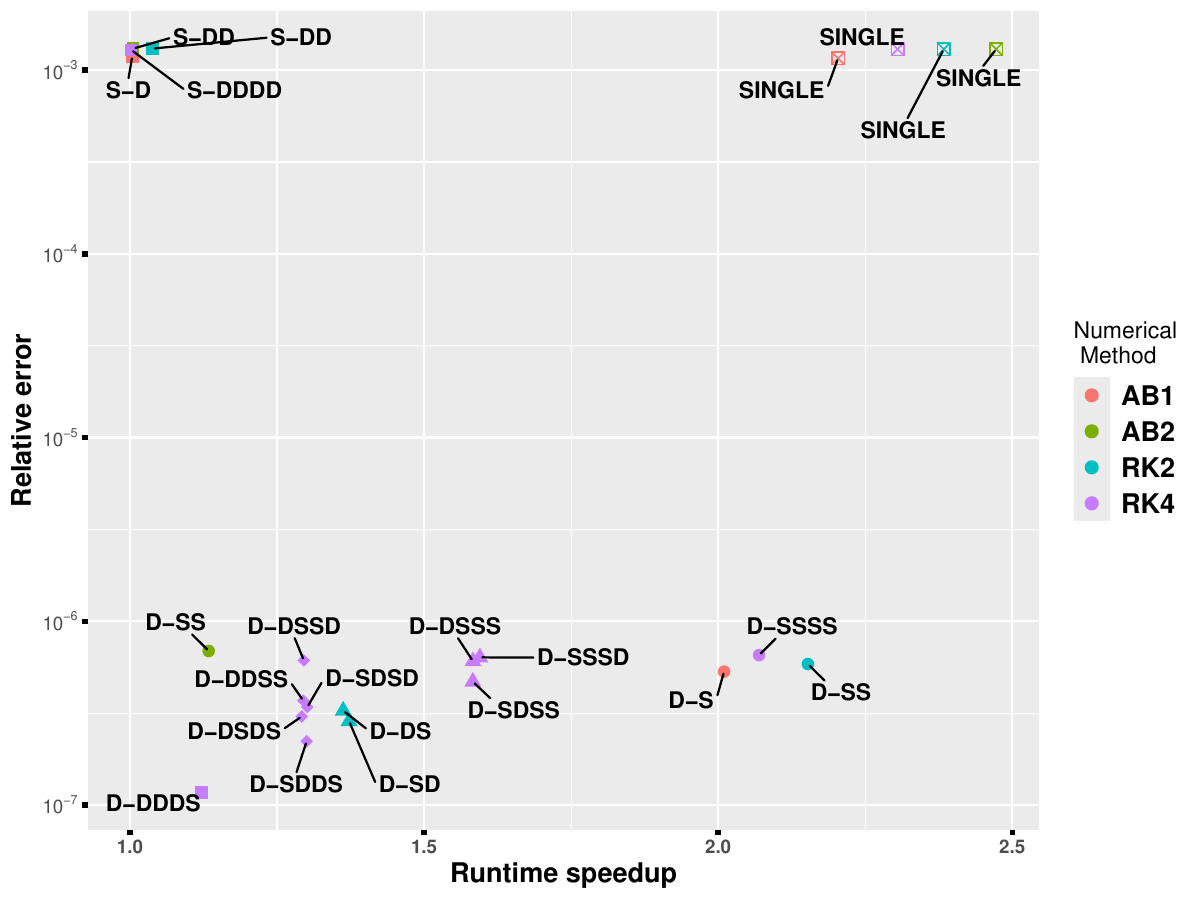}
\caption{Numerical results for benchmark $1$, with size $10^{4}$, running in sequential with the timestep $12\times 10^{-4}.$ The runtime speedup against accuracy.}
\label{BEN1__SEQ_DOUBLE}
\end{figure}
\begin{rmk}
In this paper, our objective is not to design a general mixed precision framework with iterative refinement, but a lightweight strategy tailored to explicit ODE solvers arising in biological models. For this class of problems, double precision accumulation is sufficient to preserve the required accuracy without an additional refinement step.  The double precision accumulation is the key ingredient. This is in fact the main message of the paper: most computations can be safely performed in single precision, provided that sensitive reductions and updates are accumulated in double precision.
The memory traffic plays an important role in the runtime speedup. The observed speedups stem from both the higher FP32 compute throughput and the reduced memory footprint, which improves bandwidth utilization and cache efficiency. Speedups exceeding $2\times$ are explained by the combined effects of reduced memory traffic, improved cache behavior, and the architectural imbalance between FP32 and FP64 throughput. In our experiments, we did not observe additional stability restrictions when using mixed precision compared to full double precision. The main effect was on the error magnitude rather than stability.
\end{rmk}

\section{Conclusions}
\label{Sec:Con}
The purpose of this paper is to show how lowering the arithmetic precisions, in some evaluations of the right-hand side (RHS) of the ODEs, within explicit numerical methods coded in either sequential or parallel, could accelerate the computation of the solution of the ODEs without impacting its accuracy. We provided theoretical results highlighting the efficiency of performing some portions of the numerical methods in a lower precision. These  results are
validated by numerical tests on two large ODEs that model biological systems. The numerical tests show that lowering the arithmetic precision can accelerate computations by up to $2.5$ times compared to the full double precision method, all while preserving the same high level of accuracy. Our, theoretical and numerical, results indicate that the most effective mixed method occurs when all stages of the numerical method are performed in single precision, while the final result being accumulated in double precision.  While our numerical tests primarily focused on large systems of ODEs, our mixed methods can also be applied to smaller systems that require either a large number of integration steps or when these systems are running in sequential (on a Laptop), as well as to scenarios where many small systems are solved simultaneously. In our numerical tests, we only used two kind of precision—single and double—but our mixed methods can also accommodate other types of precisions, such as half and quadruple. Using half precision in the mixed methods instead of single precision can be up to $4$ times faster than the numerical methods running in double precision, and $8$ times faster than the numerical method running in quadruple precision while maintaining high accuracy.
In our future work, we are interested in using mixed precision methods on GPU clusters to evaluate the transition from CPU to GPU, with a focus on the performance and portability of these methods.

\section*{Data availability}
The codes and data utilized in the numerical tests can be found at the following link:

 https://doi.org/10.6084/m9.figshare.27290751.v1

\section*{Acknowledgments}
We would like to acknowledge the assistance of volunteers in putting
together this example manuscript and supplement. Experiments presented in this paper were carried out using the Grid'5000 testbed, supported by a scientific interest group hosted by Inria and including CNRS, RENATER and several Universities as well as other organizations (see \url{https://www.grid5000.fr}).

\bibliographystyle{cas-model2-names}
\bibliography{references}

\begin{thebibliography}{29}
\expandafter\ifx\csname natexlab\endcsname\relax\def\natexlab#1{#1}\fi
\providecommand{\url}[1]{\texttt{#1}}
\providecommand{\href}[2]{#2}
\providecommand{\path}[1]{#1}
\providecommand{\DOIprefix}{doi:}
\providecommand{\ArXivprefix}{arXiv:}
\providecommand{\URLprefix}{URL: }
\providecommand{\Pubmedprefix}{pmid:}
\providecommand{\doi}[1]{\href{http://dx.doi.org/#1}{\path{#1}}}
\providecommand{\Pubmed}[1]{\href{pmid:#1}{\path{#1}}}
\providecommand{\bibinfo}[2]{#2}
\ifx\xfnm\relax \def\xfnm[#1]{\unskip,\space#1}\fi
%Type = Article
\bibitem[{Abdelfattah et~al.(2021)Abdelfattah, Anzt, Boman, Carson, Cojean,
  Dongarra, Fox, Gates, Higham, Li, Loe, Luszczek, Pranesh, Rajamanickam,
  Ribizel, Smith, Swirydowicz, Thomas, Tomov, Tsai and Yang}]{AA2021}
\bibinfo{author}{Abdelfattah, A.}, \bibinfo{author}{Anzt, H.},
  \bibinfo{author}{Boman, E.G.}, \bibinfo{author}{Carson, E.},
  \bibinfo{author}{Cojean, T.}, \bibinfo{author}{Dongarra, J.},
  \bibinfo{author}{Fox, A.}, \bibinfo{author}{Gates, M.},
  \bibinfo{author}{Higham, N.J.}, \bibinfo{author}{Li, X.S.},
  \bibinfo{author}{Loe, J.}, \bibinfo{author}{Luszczek, P.},
  \bibinfo{author}{Pranesh, S.}, \bibinfo{author}{Rajamanickam, S.},
  \bibinfo{author}{Ribizel, T.}, \bibinfo{author}{Smith, B.F.},
  \bibinfo{author}{Swirydowicz, K.}, \bibinfo{author}{Thomas, S.},
  \bibinfo{author}{Tomov, S.}, \bibinfo{author}{Tsai, Y.M.},
  \bibinfo{author}{Yang, U.M.}, \bibinfo{year}{2021}.
\newblock \bibinfo{title}{A survey of numerical linear algebra methods
  utilizing mixed-precision arithmetic}.
\newblock \bibinfo{journal}{Int. J. High Perform. Comput. Appl.}
  \bibinfo{volume}{35}, \bibinfo{pages}{344–369}.
%Type = Article
\bibitem[{Abdulah et~al.(2020)Abdulah, Ltaief, Sun, Genton and
  Keyes}]{ALSGK2019}
\bibinfo{author}{Abdulah, S.}, \bibinfo{author}{Ltaief, H.},
  \bibinfo{author}{Sun, Y.}, \bibinfo{author}{Genton, M.G.},
  \bibinfo{author}{Keyes, D.E.}, \bibinfo{year}{2020}.
\newblock \bibinfo{title}{Geostatistical modeling and prediction using
  mixed-precision tile cholesky factorization.}
\newblock \bibinfo{journal}{CoRR} \bibinfo{volume}{abs/2003.05324}.
%Type = Article
\bibitem[{Ackmann et~al.(2022)Ackmann, Dueben, Palmer and
  Smolarkiewicz}]{ADP2022}
\bibinfo{author}{Ackmann, J.}, \bibinfo{author}{Dueben, P.D.},
  \bibinfo{author}{Palmer, T.N.}, \bibinfo{author}{Smolarkiewicz, P.K.},
  \bibinfo{year}{2022}.
\newblock \bibinfo{title}{Mixed-precision for linear solvers in global
  geophysical flows}.
\newblock \bibinfo{journal}{Advances in Modeling Earth Systems} .
%Type = Article
\bibitem[{Al-Sayed-Ali et~al.(2025)Al-Sayed-Ali, Bernard, Marzorati and
  Rouzaud-Cornabas}]{ABMC2025}
\bibinfo{author}{Al-Sayed-Ali, M.}, \bibinfo{author}{Bernard, S.},
  \bibinfo{author}{Marzorati, A.}, \bibinfo{author}{Rouzaud-Cornabas, J.},
  \bibinfo{year}{2025}.
\newblock \bibinfo{title}{Mixed precision implicit numerical schemes for
  systems of ordinary differential equations}.
\newblock \bibinfo{journal}{Numerical Algorithms} .
%Type = Article
\bibitem[{Balos et~al.(2023)Balos, Roberts and Gardner}]{BRG2023}
\bibinfo{author}{Balos, C.J.}, \bibinfo{author}{Roberts, S.},
  \bibinfo{author}{Gardner, D.J.}, \bibinfo{year}{2023}.
\newblock \bibinfo{title}{Leveraging mixed precision in exponential time
  integration methods}.
\newblock \bibinfo{journal}{2023 IEEE High Performance Extreme Computing
  Conference (HPEC)} , \bibinfo{pages}{1--8}.
%Type = Article
\bibitem[{Burnett et~al.(2021)Burnett, Gottlieb, Grant and
  Heryudono}]{BGGH2021}
\bibinfo{author}{Burnett, B.}, \bibinfo{author}{Gottlieb, S.},
  \bibinfo{author}{Grant, Z.}, \bibinfo{author}{Heryudono, A.},
  \bibinfo{year}{2021}.
\newblock \bibinfo{title}{Performance evaluation of mixed-precision runge-kutta
  methods}.
\newblock \bibinfo{journal}{IEEE High Performance Extreme Computing Conference
  (HPEC)} , \bibinfo{pages}{1--6}.
%Type = Article
\bibitem[{Butcher(1996)}]{B1996}
\bibinfo{author}{Butcher, J.C.}, \bibinfo{year}{1996}.
\newblock \bibinfo{title}{A history of runge-kutta methods}.
\newblock \bibinfo{journal}{Applied Numerical Mathematics}
  \bibinfo{volume}{20}, \bibinfo{pages}{247--260}.
%Type = Article
\bibitem[{Butcher(2003)}]{B2003}
\bibinfo{author}{Butcher, J.C.}, \bibinfo{year}{2003}.
\newblock \bibinfo{title}{Numerical methods for ordinary differential
  equations}.
\newblock \bibinfo{journal}{John Wiley} .
%Type = Article
\bibitem[{Das et~al.(2018)Das, Mellempudi, Mudigere, Kalamkar, Avancha,
  Banerjee, Sridharan, Vaidyanathan, Kaul, Georganas, Heinecke, Dubey, Corbal,
  Shustrov, Dubtsov, Fomenko and Pirogov}]{DMMK2018}
\bibinfo{author}{Das, D.}, \bibinfo{author}{Mellempudi, N.},
  \bibinfo{author}{Mudigere, D.}, \bibinfo{author}{Kalamkar, D.D.},
  \bibinfo{author}{Avancha, S.}, \bibinfo{author}{Banerjee, K.},
  \bibinfo{author}{Sridharan, S.}, \bibinfo{author}{Vaidyanathan, K.},
  \bibinfo{author}{Kaul, B.}, \bibinfo{author}{Georganas, E.},
  \bibinfo{author}{Heinecke, A.}, \bibinfo{author}{Dubey, P.},
  \bibinfo{author}{Corbal, J.}, \bibinfo{author}{Shustrov, N.},
  \bibinfo{author}{Dubtsov, R.}, \bibinfo{author}{Fomenko, E.},
  \bibinfo{author}{Pirogov, V.O.}, \bibinfo{year}{2018}.
\newblock \bibinfo{title}{Mixed precision training of convolutional neural
  networks using integer operations}.
\newblock \bibinfo{journal}{CoRR} \bibinfo{volume}{abs/1802.00930}.
%Type = Article
\bibitem[{Düben et~al.(2017)Düben, Subramanian, Dawson and Palmer}]{DSDP2017}
\bibinfo{author}{Düben, P.}, \bibinfo{author}{Subramanian, A.},
  \bibinfo{author}{Dawson, A.}, \bibinfo{author}{Palmer, T.},
  \bibinfo{year}{2017}.
\newblock \bibinfo{title}{A study of reduced numerical precision to make
  superparameterization more competitive using a hardware emulator in the
  openifs model}.
\newblock \bibinfo{journal}{Journal of Advances in Modeling Earth Systems} .
%Type = Article
\bibitem[{Düben and Palmer(2020)}]{KDP2020}
\bibinfo{author}{Düben, P.D.}, \bibinfo{author}{Palmer, T.N.},
  \bibinfo{year}{2020}.
\newblock \bibinfo{title}{Number formats, error mitigation, and scope for
  16‐bit arithmetics in weather and climate modeling analyzed with a shallow
  water model}.
\newblock \bibinfo{journal}{Journal of Advances in Modeling Earth Systems}
  \bibinfo{volume}{12}.
%Type = Article
\bibitem[{El~Cheikh et~al.(2017)El~Cheikh, Bernard and El~Khatib}]{EBE2017}
\bibinfo{author}{El~Cheikh, R.}, \bibinfo{author}{Bernard, S.},
  \bibinfo{author}{El~Khatib, N.}, \bibinfo{year}{2017}.
\newblock \bibinfo{title}{{A multiscale modelling approach for the regulation
  of the cell cycle by the circadian clock}}.
\newblock \bibinfo{journal}{{Journal of Theoretical Biology}}
  \bibinfo{volume}{426}, \bibinfo{pages}{117--125}.
\newblock \URLprefix \url{https://hal.science/hal-01561617}.
%Type = Article
\bibitem[{El~Cheikh~R(2014)}]{EBE2014}
\bibinfo{author}{El~Cheikh~R, Bernard~S, E.K.N.}, \bibinfo{year}{2014}.
\newblock \bibinfo{title}{Modeling circadian clock-cell cycle interaction
  effects on cell population growth rates}.
\newblock \bibinfo{journal}{J Theor Biol.} .
%Type = Article
\bibitem[{Grant(2022)}]{G2022}
\bibinfo{author}{Grant, Z.J.}, \bibinfo{year}{2022}.
\newblock \bibinfo{title}{Perturbed runge-kutta methods for mixed precision
  applications}.
\newblock \bibinfo{journal}{J Sci Comput} \bibinfo{volume}{6}.
%Type = Article
\bibitem[{Hairer et~al.(2008)Hairer, McLachlan and Razakarivony}]{HMR2008}
\bibinfo{author}{Hairer, E.}, \bibinfo{author}{McLachlan, R.I.},
  \bibinfo{author}{Razakarivony, A.}, \bibinfo{year}{2008}.
\newblock \bibinfo{title}{Achieving brouwer’s law with implicit runge–kutta
  methods}.
\newblock \bibinfo{journal}{BIT} \bibinfo{volume}{48},
  \bibinfo{pages}{231–243}.
%Type = Book
\bibitem[{Hairer et~al.(2000)Hairer, N{\o}rsett and Wanner}]{HNW87}
\bibinfo{author}{Hairer, E.}, \bibinfo{author}{N{\o}rsett, S.},
  \bibinfo{author}{Wanner, G.}, \bibinfo{year}{2000}.
\newblock \bibinfo{title}{Solving Ordinary Differential Equations {I} Nonstiff
  problems}.
\newblock \bibinfo{edition}{Second} ed., \bibinfo{publisher}{Springer},
  \bibinfo{address}{Berlin}.
%Type = Book
\bibitem[{Hairer and Wanner(1996)}]{HW96}
\bibinfo{author}{Hairer, E.}, \bibinfo{author}{Wanner, G.},
  \bibinfo{year}{1996}.
\newblock \bibinfo{title}{Solving Ordinary Differential Equations II. Stiff and
  Differential-Algebraic Problems}. volume~\bibinfo{volume}{14}.
\newblock \bibinfo{publisher}{Springer Verlag Series in Comput. Math.}
%Type = Article
\bibitem[{Higham and Mary(2022)}]{HM2022}
\bibinfo{author}{Higham, N.J.}, \bibinfo{author}{Mary, T.},
  \bibinfo{year}{2022}.
\newblock \bibinfo{title}{Mixed precision algorithms in numerical linear
  algebra}.
\newblock \bibinfo{journal}{Acta Numerica} \bibinfo{volume}{31},
  \bibinfo{pages}{347–414}.
%Type = Article
\bibitem[{van~der Houwen and de~Swart(1997)}]{VS97}
\bibinfo{author}{van~der Houwen, P.J.}, \bibinfo{author}{de~Swart, J.J.B.},
  \bibinfo{year}{1997}.
\newblock \bibinfo{title}{Triangularly implicit iteration methods for ode-ivp
  solvers}.
\newblock \bibinfo{journal}{SIAM Journal on Scientific Computing}
  \bibinfo{volume}{18}, \bibinfo{pages}{41--55}.
%Type = Article
\bibitem[{Kelley(2022)}]{K2022}
\bibinfo{author}{Kelley, C.T.}, \bibinfo{year}{2022}.
\newblock \bibinfo{title}{Newton's method in mixed precision}.
\newblock \bibinfo{journal}{SIAM Review} \bibinfo{volume}{64},
  \bibinfo{pages}{191--211}.
%Type = Article
\bibitem[{Klower et~al.(2022)Klower, Hatfield, Croci, Duben and
  Palmer}]{KDP2021}
\bibinfo{author}{Klower, M.}, \bibinfo{author}{Hatfield, S.},
  \bibinfo{author}{Croci, M.}, \bibinfo{author}{Duben, P.},
  \bibinfo{author}{Palmer, T.}, \bibinfo{year}{2022}.
\newblock \bibinfo{title}{Fluid simulations accelerated with 16 bits:
  approaching 4x speedup on a64fx by squeezing shallowwaters.jl into float16}.
\newblock \bibinfo{journal}{Journal of Advances in Modeling Earth Systems}
  \bibinfo{volume}{14}.
%Type = Article
\bibitem[{Lima and Buckwar(2015)}]{LB2015}
\bibinfo{author}{Lima, P.M.}, \bibinfo{author}{Buckwar, E.},
  \bibinfo{year}{2015}.
\newblock \bibinfo{title}{Numerical solution of the neural field equation in
  the two-dimensional case}.
\newblock \bibinfo{journal}{SIAM Journal on Scientific Computing}
  \bibinfo{volume}{37}, \bibinfo{pages}{B962--B979}.
\newblock \DOIprefix\doi{10.1137/15M1022562}.
%Type = Article
\bibitem[{M.~Croci(2022)}]{CS2022}
\bibinfo{author}{M.~Croci, G.R.S.}, \bibinfo{year}{2022}.
\newblock \bibinfo{title}{Mixed-precision explicit stabilized runge–kutta
  methods for single- and multi-scale differential equations}.
\newblock \bibinfo{journal}{Journal of Computational Physics}
  \bibinfo{volume}{464}, \bibinfo{pages}{111349}.
%Type = Article
\bibitem[{Mellempudi et~al.(2019)Mellempudi, Srinivasan, Das and
  Kaul}]{MSDK2019}
\bibinfo{author}{Mellempudi, N.}, \bibinfo{author}{Srinivasan, S.},
  \bibinfo{author}{Das, D.}, \bibinfo{author}{Kaul, B.}, \bibinfo{year}{2019}.
\newblock \bibinfo{title}{Mixed precision training with 8-bit floating point.}
\newblock \bibinfo{journal}{CoRR} \bibinfo{volume}{abs/1905.12334}.
%Type = Article
\bibitem[{Paxton et~al.(2022)Paxton, Chantry, Klöwer, Saffin and
  Palmer}]{PCKSP2022}
\bibinfo{author}{Paxton, E.A.}, \bibinfo{author}{Chantry, M.},
  \bibinfo{author}{Klöwer, M.}, \bibinfo{author}{Saffin, L.},
  \bibinfo{author}{Palmer, T.}, \bibinfo{year}{2022}.
\newblock \bibinfo{title}{Climate modeling in low precision: Effects of both
  deterministic and stochastic rounding}.
\newblock \bibinfo{journal}{Journal of Climate} \bibinfo{volume}{35},
  \bibinfo{pages}{1215 -- 1229}.
%Type = Article
\bibitem[{S.(1977)}]{A1977}
\bibinfo{author}{S., A.}, \bibinfo{year}{1977}.
\newblock \bibinfo{title}{Dynamics of pattern formation in lateral-inhibition
  type neural fields}.
\newblock \bibinfo{journal}{Biological Cybernetics} .
%Type = Article
\bibitem[{Váňa et~al.(2017)Váňa, Düben, Lang, Palmer, Leutbecher, Salmond
  and Carver}]{VDLSC2017}
\bibinfo{author}{Váňa, F.}, \bibinfo{author}{Düben, P.},
  \bibinfo{author}{Lang, S.}, \bibinfo{author}{Palmer, T.},
  \bibinfo{author}{Leutbecher, M.}, \bibinfo{author}{Salmond, D.},
  \bibinfo{author}{Carver, G.}, \bibinfo{year}{2017}.
\newblock \bibinfo{title}{Single precision in weather forecasting models: An
  evaluation with the ifs}.
\newblock \bibinfo{journal}{Monthly Weather Review} \bibinfo{volume}{145},
  \bibinfo{pages}{495 -- 502}.
\newblock \DOIprefix\doi{10.1175/MWR-D-16-0228.1}.
%Type = Article
\bibitem[{Wilson and Cowan(1972)}]{WC1972}
\bibinfo{author}{Wilson, H.R.}, \bibinfo{author}{Cowan, J.D.},
  \bibinfo{year}{1972}.
\newblock \bibinfo{title}{Excitatory and inhibitory interactions in localized
  populations of model neurons}.
\newblock \bibinfo{journal}{Biophysical Journal} \bibinfo{volume}{12},
  \bibinfo{pages}{1--24}.
%Type = Article
\bibitem[{Yingqi et~al.(2022)Yingqi, Takeshi, Linjie and Takeshi}]{ZFZI2022}
\bibinfo{author}{Yingqi, Z.}, \bibinfo{author}{Takeshi, F.},
  \bibinfo{author}{Linjie, Z.}, \bibinfo{author}{Takeshi, I.},
  \bibinfo{year}{2022}.
\newblock \bibinfo{title}{Numerical investigation into the mixed precision
  gmres(m) method using fp64 and fp32}.
\newblock \bibinfo{journal}{Journal of Information Processing}
  \bibinfo{volume}{30}.

\end{thebibliography}

\vspace{-0.5 cm}

\section[here]{Appendix $1$}\nonumber

\vspace{-0.1 cm}

\begin{longtable}{|p{0.2\linewidth}p{0.1\linewidth}p{0.5\linewidth}|}
%\begin{center}
%\begin{tabular}{|c|c|c|}
\hline\mbox{Parameter}&  \mbox{Value}&\mbox{Unit}
\\ \hline
$k_s$&	  	0.1	&		unitless \\
$\eta$	&	0.01	&			per cell \\
$k_d$	&	0.0	&		$h^{-1}$ \\
$C$	&	0.4	&		nM\\
$b_{bmal0}$	&	0.0	&		nM\\
$p_0$	&	4		&	 \\
$\nu_{1b}$	&        9.0		&		nM $h^{-1}$  \\
$k_{1b}$	&	1.0		&		nM \\
$k_{1d}$	&	0.12		&		$h^{-1}$  \\
$k_{1i}$	&	0.56		&		nM \\
$k_{2b}$	&	0.3		&		$nM^-1$ $h^{-1}$  \\
$k_{2d}$	&	0.05		&		$h^{-1}$  \\
$k_{2t}$	&	0.24		&		$h^{-1}$  \\
$k_{3t}$	&	0.02		&		$h^{-1}$  \\
$q$	&	2		&		unitless \\
$k_{3d}$	&	0.12		&		$h^{-1}$  \\
$\nu_{4b}$	&	3.6		&		$nM^-1$ $h^{-1}$  \\
$r_0$	&	3		&		unitless \\
$k_{4b}$	&	2.16		&		$nM^-1$ $h^{-1}$  \\
$k_{4d}$	&	0.75		&		$h^{-1}$  \\
$k_{5b}$	&	0.24		&		$h^{-1}$  \\
$k_{5d}$	&	0.06		&		$h^{-1}$  \\
$k_{5t}$	&	0.45		&		$h^{-1}$  \\
$k_{6t}$	&	0.06		&		$h^{-1}$  \\
$k_{6d}$	&	0.12		&		$h^{-1}$  \\
$k_{6a}$	&	0.09		&		$h^{-1}$  \\
$k_{7a}$	&	0.003		&		$h^{-1}$  \\
$k_{7d}$	&	0.09		&		$h^{-1}$  \\
$\tau_{0}$	&	1.0		&		unitless\\
std$\_$circadian$\_$clock &	0.05	&		unitless \\
$c$		&	0.01	&		nM \\
$k_{impf}$		&	4.0	&		$h^{-1}$ \\
$k_{0mpf}$		&	6	&		$h^{-1}$ \\
$k_{1mpf}$		&	0.05	&		nM \\
$s$		&	20.0	&		nM \\
$d_{wee1}$		&	5.0	&		$h^{-1}$ \\
$n_0$		&	2	&		unitless \\
$k_{actw}$		&	1.0	&		$h^{-1}$ \\
$d_{w1}$		&	1.0	&		nM \\
$c_w$		&	0.5	&		nM \\
$k_{inactw}$		&	200.0	&		$h^{-1}$ \\
$k_{1wee1}$		&	0.5	&		nM \\
$d_{w2}$		&	1.0	&		$h^{-1}$ \\
$k_{act}$		&	0.01	&		$h^{-1}$ \\
target$\_$period	&	20.0	&		intrinsic period of the cell cycle in hours \\
std$\_$cell$\_$cycle	&	0.1	&		relative variability of the cell cycle periods \\ \hline
%\end{tabular}
\caption{The parameters for benchmark $1$}\label{Para_Cira}
%\end{center}
\end{longtable}
The initial solution $y_0$ is given by :
$$
\begin{aligned}
&y_1^{(i)}(0) = 0.1 \times \mbox{rand}^{(i)}\\
&y_2^{(i)}(0) =  0.2 	\times \mbox{rand}^{(i)}\\
&y_3^{(i)}(0) =  1.8 	\times \mbox{rand}^{(i)}\\
&y_4^{(i)}(0) =  0.4 	\times \mbox{rand}^{(i)}\\
&y_5^{(i)}(0) =  0.5 	\times \mbox{rand}^{(i)} \\
&y_6^{(i)}(0) =  0.6 	\times \mbox{rand}^{(i)}\\
&y_7^{(i)}(0) =  0.1 	\times \mbox{rand}^{(i)}\\
&y_8^{(i)}(0) =  0.1 	\times \mbox{rand}^{(i)}\\
&y_9^{(i)}(0) =  0.1 	\times \mbox{rand}^{(i)}\\
&y_{10}^{(i)}(0) =  0.1 	\times \mbox{rand}^{(i)}
\end{aligned}
$$
for $i=1,\ldots,d.$ Parameters $\lambda_0, \lambda$ and $\tau$ are given by
$$
\begin{aligned}
&\lambda_0 = 97.4/target\_period \\
&\lambda = \mbox{rand}(\lambda_0,std\_cell\_cycle\times\lambda_0), \\
&\tau = \mbox{rand}(\tau_0,std\_circadian\_clock\times\tau_0).
\end{aligned}
$$

The vector rand is computed by the subroutine "$call\  random\_number(\mbox{rand})$."

\end{document}